\documentclass[11pt, reqno]{amsart}
\usepackage{version}
\usepackage{amssymb}
\usepackage{tikz-cd} 
\newtheorem{theorem}{Theorem}[section]
\newtheorem{lemma}[theorem]{Lemma}
\newtheorem{corollary}[theorem]{Corollary}
\newtheorem{prop}[theorem]{Proposition}

\theoremstyle{definition}
\newtheorem{definition}[theorem]{Definition}

\newtheorem{problem}[theorem]{Problem}
\newtheorem{remark}[theorem]{Remark}

\numberwithin{equation}{section}

\begin{document}
\title[]{Injective Hyperbolicity for quotients of balls and polydisks}
\author[Forn\ae ss]{J. E. Forn\ae ss}
\author[Trybula]{M. Trybula}
\author[Wold]{E. F. Wold}
\address{E. F. Wold: Department of Mathematics\\
University of Oslo\\
Postboks 1053 Blindern, NO-0316 Oslo, Norway}\email{erlendfw@math.uio.no}

%
%
\date{\today}

\begin{abstract}
In this article we study the injective Kobayashi metric on complex surfaces. 
\end{abstract}

\maketitle

\section{Introduction}

For complex manifolds $Y$ and $X$ we let $\mathcal O(Y,X)$ denote 
the set of holomorphic maps $f:Y\rightarrow X$, and we let 
$\mathcal O_\iota(Y,X)$ be the set of elements $f\in \mathcal O(Y,X)$
such that $f$ is injective. We let $\triangle$ denote the unit disk in the complex plane. 
\begin{definition}
Let $X$ be a complex manifold. For a point $(x,v)\in TX$ we set
\begin{equation}\label{def:kob}
\omega^X_K(x,v):=\inf\{\frac{1}{\lambda}:f\in\mathcal O(\triangle,X), f(0)=x, df(0)(1)=\lambda v\}, 
\end{equation}
and we set 
\begin{equation}\label{def:ikob}
\omega^X_\iota(x,v):=\inf\{\frac{1}{\lambda}:f\in\mathcal O_\iota(\triangle,X), f(0)=x, df(0)(1)=\lambda v\}.
\end{equation}
Then $\omega^X_K$ is the familiar (infinitesimal) Kobayashi metric, and we will call $\omega^X_\iota$ the \emph{injective Kobayashi metric}.
\end{definition}

\begin{remark}
Upon finishing a preprint of the current article we were made aware of the fact that the injective Kobayashi
metric was already introduced in \cite{Hahn}, and that it already appeared in the one-dimensional case in \cite{Siu}.
In \cite{JarnickiPflug} the corresponding object(s) are referred to as Hahn functions/metrics. Furthermore, these 
objects were studied in \cite{Jarnicki1} and \cite{Jarnicki2}, and Theorem \ref{thm:bidisk1} below may be proved 
by the methods in \cite{Jarnicki1} (where the corresponding result was proved for non-simply connected hyperbolic
domains in $\mathbb C$  - see also \cite{JarnickiPflug} Chapter 8), or as an application of the result therein.  
\end{remark}

The main problem is the following. 
\begin{problem}\label{prob:main}
For which 2-dimensional complex manifolds $X$ do we have that $\omega_\iota^X=\omega^X$?
\end{problem}

In complex dimension one, i.e., in the the case that $X$ is a Riemann surface, the corresponding problem 
is quite simple. If $X$ is hyperbolic the metrics coincide if and only if $X$ is the unit disk (see \cite{SibonyWold} where the 
injective Kobayashi metric was introduced on foliations), 
if $X=\mathbb C$ or $X=\mathbb P^1$ both metrics vanish identically, and if $X=\mathbb C^*$ or 
$X$ is a torus, the metrics are different due to the Koebe $\frac{1}{4}$-theorem. Furthermore, in complex 
dimension larger than 2, the metrics always coincide, see \cite{Overholt}. 

\medskip

In this note, we will
here give some initial results, focusing on quotients of balls and bi-disks.

\begin{theorem}\label{thm:main}
Let $\Gamma\subset\mathrm{Aut_{hol}}(\mathbb B^2)$ be a Kleinian group, set $X:=\mathbb B^2/\Gamma$,
and assume that $X$ is compact.
Then $\omega^X_K\neq\omega^X_\iota$.
\end{theorem}
Theorem \ref{thm:main} will be proved in Section \ref{sec:ball}. \

\medskip

The following result is essentially due to Jarnicki \cite{Jarnicki1}. 

\begin{theorem}\label{thm:bidisk1}(Jarnicki \cite{Jarnicki1})
Let $Y_1$ and $Y_2$ be compact hyperbolic Riemann surfaces, and set $X:=Y_1\times Y_2$.
Then $\omega^X_K\neq\omega^X_\iota$.
\end{theorem}

The last theorem will be a consequence of a more general result proved in Section \ref{sec:bidisk}, where 
we will also construct non-trivial examples where the two metrics coincide.

\section{Preliminaries}

\subsection{Definitions}

Throughout this article $\triangle$ will denote the unit disk in the complex plane $\mathbb C$, and
$\triangle^2$ will denote the unit bidisk in $\mathbb C^2$. For any domain $\Omega\subset\mathbb C^n$
we let $b\Omega$ denote its topological boundary.  For a point $p\in\mathbb C^2$ and $\delta>0$, 
we let $B_\delta(p)$ denote the ball of radius $\delta$ centred at $p$.  \

Recall that for $X=\triangle$ we have that 
$$
\omega_K^\triangle(z,v)=\omega_P(z,v)=\frac{|v|}{1-|z|^2},
$$
where $\omega_P$ denotes the Poincar\'{e} metric. Equipped with this metric, the 
holomorphic automorphism group $\mathrm{Aut_{hol}}(\triangle)$ of the unit disk, is 
the group of orientation preserving isometries of $\triangle$, and any Riemann surface
of hyperbolic type is the quotient of $\triangle$ by a Fuchsian sub-group $\Gamma$.

\begin{definition}
Let $\Gamma\subset\mathrm{Aut_{hol}}(\triangle)$ be a sub-group. 
We til call $\Gamma$ a \emph{Fuchsian group} if $\Gamma$ acts properly discontinuously on $\triangle$, i.e.,
if for every point $z\in\mathbb\triangle$ there is an open set $U$ containing $z$ such that 
if $\phi\in\Gamma$ and if $\phi(U)\cap U\neq\emptyset$, then $\phi=\mathrm{id}$. If $\triangle$
is replaced by $\mathbb B^2$ or $\triangle^2$ we will call such a group a Kleinian group. 
\end{definition}

Recall that any element $\phi\in\mathrm{Aut_{hol}}(\triangle^2)$ is of the form 
$\phi=(\varphi_1,\varphi_2)$ with $\varphi_j\in\mathrm{Aut_{hol}}(\triangle)$.
The elements  $\varphi\in\mathrm{Aut_{hol}}(\triangle)$ are classified into 
one of three types: 
\begin{itemize}
\item[(1)] \emph{hyperbolic} if $\varphi$ has precisely two distinct fixed points on $\overline\triangle$,
and both are contained in $b\triangle$,  
\item[(2)] \emph{parabolic} if  $\varphi$ has precisely one fixed points on $\overline\triangle$, and it 
is contained in $b\triangle$, 
\item[(3)] \emph{elliptic} if  $\varphi$ has precisely one fixed points on $\overline\triangle$, and it 
is contained in $\triangle$.
\end{itemize}
Elements of $\mathrm{Aut_{hol}}(\mathbb B^2)$ are classified correspondingly, were in (1)-(3) we
replace $\triangle$ by $\mathbb B^2$. 

\medskip

Clearly, a Kleinian/Fuchsian group cannot contain any elements of elliptic type, so we will here consider 
hyperbolic and parabolic automorphisms. Recall that a hyperbolic automorphism of $\triangle$
is conjugate to an automorphism
\begin{equation}\label{diskhyp}
\varphi(z) = \frac{z+r}{1+rz}
\end{equation}
with $0<r<1$. For parabolic automorphisms it is more convenient to identify $\triangle$ with 
the upper half plane $H$, and there any parabolic automorphism is conjugate to 
\begin{equation}\label{halfplanepara}
\mbox{either }\varphi^+(z) = z+1 \mbox{ or } \varphi^-(z) = z-1.
\end{equation}
In dimension two, any hyperbolic automorphism of $\mathbb B^2$ is conjugate to an
automorphism 
\begin{equation}\label{ballhyp}
\phi(z,w) = (\frac{z+r}{1+rz}, e^{i\theta}\frac{\sqrt{1-r^2}}{1+rz}w)
\end{equation}
for $0<r<1$ and $\theta\in [0,2\pi)$. For parabolic automorphisms it is more convenient to identify
$\mathbb B^2$ with the Siegel upper half plane 
$$
H_2 = \{(z,w)\in\mathbb C^2:\mathrm{Im}(w)>|z|^2\}.
$$
In this case, any parabolic automorphism of $H_2$ is conjugate (after possibly passing to inverses) to one of the following two types 
\begin{equation}\label{siegelpar1}
 \phi(z,w)=(e^{i\theta}z,w+1) 
\end{equation}
or 
\begin{equation}\label{siegelpar2}
\phi(z,w)=(z-i,w-2z + i). 
\end{equation}

\medskip
%

\medskip

\subsection{Extremal maps in $\mathbb B^2$.}\label{extmaps:ball}

For a point $p\in\mathbb B^2$ and a tangent vector $v$, and extremal map $f:\triangle\rightarrow\mathbb B^2$
is a map such that $f(0)=p$, and $\omega_K(p,v)=\frac{1}{|f'(0)|}$. 
If we want to determine all extremal maps for a point $p$, since $\mathrm{Aut_{hol}}(\mathbb B^2)$
acts transitively on $\mathbb B^2$, it suffices to consider the case $p=0$. And since the isotropy group at the origin acts transitively on directions, it suffices to consider the case $v=(1,0)$. Then using Schwarz Lemma, it follows that
the map $f(z)=(z,0)$ is extremal, and furthermore that $f$ is the unique extremal map for $\omega_K(0,v)$.

\subsection{Extremal maps in $\triangle^2$.}\label{extmaps:bidisk} For any point $z\in\triangle$ and any vector $v\in\mathbb C^2$, 
it follows by Montel's Theorem that there exists a map $f:\triangle\rightarrow\triangle^2$
such that $f(0)=z$ and $\omega^{\triangle^2}_K(z,v)=\frac{1}{|f'(0)|}$. Consider 
such extremal maps for $z=0$ and $v=(1,\xi)$ with $|\xi|\leq 1$. Then 
a natural candidate for an extremal map is the map $f(z)=(z,z\cdot\xi)$. This map 
clearly has a left inverse $\psi:\triangle^2\rightarrow\triangle$, namely $\psi(z_1,z_2)=z_1$, 
and using this it is clear from the Schwarz Lemma that $f$ indeed is an extremal map. 
Moreover, if $|\xi|=1$ the map $f$ is the unique extremal map, which can be seen 
by applying Schwarz Lemma after projecting to the diagonal.

\section{The general strategy}\label{sec:strategy}

\begin{prop}\label{prop:genstrat}
Let $\Omega\subset\mathbb C^2$ be a bounded domain, let $X$ be a complex manifold, and suppose that $\pi:\Omega\rightarrow X$
is a holomorphic covering map. Assume that $f:\triangle\rightarrow\Omega$ is a proper holomorphic embedding such that 
$f$ is a unique extremal map for the Kobayashi metric through the point $p=f(0)$ with tangent vector $v=f'(0)$. Assume further that there is a point 
$q\in Z=f(\triangle)$ and a $\delta>0$ such that 
$$
\pi(B_\delta(p)\cap Z)\cap \pi(B_\delta(q)\cap Z)=\pi(p)=\pi(q).
$$
Then $\omega^X_\iota(p,v_*)>\omega^X_K(p,v_*)$ where $v_*=\pi_*v$. Consequently $\omega^X_\iota\neq\omega^X_K$. 
\end{prop}
\begin{proof}
We will show first that if $\tilde f:\triangle\rightarrow\Omega$ is sufficiently close to 
$f$, then $\pi\circ\tilde f:\triangle\rightarrow X$ is not injective. 
Set $a=f^{-1}(q)$ and fix $\epsilon>0$ such that $\pi(f(B_\epsilon(0)))\cap\pi(f(B_\epsilon(a)))=\pi(p)$.
Define $G_j:\triangle_\epsilon\rightarrow X$ by $G_1(z)=\pi(f(z))$, and $G_2(z)=\pi(f(z+a))$.
We will show that if $\tilde G_j$ is sufficiently close to $G_j$ for $j=1,2$, then 
$G_1(\triangle_\epsilon)\cap G_2(\triangle_\epsilon)\neq\emptyset$. By passing to a local coordinate 
chart, and possibly having to decrease $\epsilon$, we may assume that $G_j(\triangle_\epsilon)\subset\mathbb C^2$, $G_j(0)=0$,
and $dG_1(0)(1)=(1,0)$, and further that $G_1(z)=(g(z), h(g(z)))$. Setting 
$H(z_1,z_2)=z_2-h(z_1)$ we now have that the function $H(G_2(z))$ has an isolated zero at the origin
on $\triangle_{\epsilon/2}$. \

Now suppose that $G^k_j\rightarrow G_j$ uniformly as $k\rightarrow\infty$ on $\triangle_\epsilon$ for $j=1,2$.
Then for sufficiently large $k$ we may write $G^k_1(z)=(g_k(z),h_k(g_k(z)))$, and we have that 
$g_k\rightarrow g, h_k\rightarrow h$, and so setting $H_k(z)=z_2-h_k(z_1)$ we have that 
$H_k\rightarrow H$ uniformly as $k\rightarrow\infty$. Then $H_k\circ G^k_2\rightarrow H\circ G_2$
uniformly as $k\rightarrow\infty$, and so by Hurwitz' Theorem $H_k\circ G^2_k$ has a zero for 
sufficiently large $k$.  This means precisely that the images of the two discs intersect.  \

To finish the proof of the proposition, let $\tilde f_i:\triangle\rightarrow X$ be a sequence of holomorphic maps 
with $f(0)=\pi(p),d\tilde f_i(0)(1)=\lambda_i^{-1}\cdot v_*$, and such that $\lambda_i\rightarrow \lambda=\omega^X_K(\pi(p),v_*)$.
Letting $f_j:\triangle\rightarrow\Omega$ be liftings such that $\tilde f_j=\pi\circ f_j$, we get by uniqueness of 
$f$ that $f_i\rightarrow f$, and by our previous conclusion we have that $f_i$ is not injective for sufficiently large 
$i$. 

\end{proof}

\section{Quotients of the unit ball}\label{sec:ball}

The simplest situation where we in our context can find extremal maps to apply Proposition \ref{prop:genstrat}, 
is found where we consider a Kleinian subgroup $\Gamma\subset\mathrm{Aut_{hol}}(\mathbb B^2)$ which 
contains at least one hyperbolic element.

\begin{theorem}\label{thm:hypelm}
Let $\Gamma\subset\mathrm{Aut_{hol}}(\mathbb B^2)$ be a Kleinian group, set $X:=\mathbb B^2/\Gamma$, 
and assume that $\Gamma$ contains at least one hyperbolic element. 
Then $\omega^X_K\neq\omega^X_\iota$.
\end{theorem}

\begin{proof}
After conjugation we may achieve that a 
hyperbolic element is of the form
$$
\phi(z,w)=(\frac{z+r}{1+rz},e^{i\theta}\frac{\sqrt{1-r^2}}{1+rz}w)
$$
with $0\leq \theta<2\pi$.  Note that, by replacing $\phi$ with a high iterate, we may 
assume that $r$ and $e^{i\theta}$ are both arbitrarily close to $1$. 
For $\alpha\in(0,\sqrt{1-r^2})$ to be determined further, consider the straight line $L_\alpha:=\{(z,\alpha): |z|^2<1-\alpha^2\}$. 
Then $\phi$ sends $L_\alpha$ to the straight line 
$$
L^\phi_\alpha = \{(z,w): w = \frac{\alpha e^{i\theta}(1-rz)}{\sqrt{1-r^2}}\}.
$$
The intersection point between $L_\alpha$ and $L^\phi_\alpha$ occurs for $z_0=\frac{1-\sqrt{1-r^2}e^{-i\theta}}{r}$.  We have 
that 

\begin{align*}
|z_0|^2 - 1 < 0 & \Leftrightarrow  (1-e^{-i\theta}\sqrt{1-r^2})\cdot (1-e^{i\theta}\sqrt{1-r^2}) - r^2 < 0\\
& \Leftrightarrow  1 - \sqrt{1-r^2}\cdot 2\cos\theta + (1-r^2) - r^2 < 0 \\
& \Leftrightarrow 2(1-r^2) - 2\sqrt{1-r^2}\cdot\cos\theta < 0 \\
& \Leftrightarrow 2\sqrt{1-r^2}(\sqrt{1-r^2}-\cos\theta)<0. 
\end{align*}
So if $r$ is close enough to 1 and if $\theta$ is close enough to $0$ we have that $z_0$ is 
in the unit disk. Then $|z_0|<\sqrt{1-\alpha^2}$ if $\alpha$ is chosen small enough. 
The conditions
in Proposition \ref{prop:genstrat} are therefore fulfilled. 
\end{proof}

\emph{Proof of Theorem \ref{thm:main}:} By Theorem \ref{thm:hypelm} it suffices to prove that if $X$ is compact, 
then $\Gamma$ contains a hyperbolic element. In fact, if $X$ is compact, we have that $\Gamma$ contains only 
hyperbolic elements. 
\begin{lemma}
Let $X=\mathbb B^2/\Gamma$ be a compact complex manifold. Then $\Gamma$ contains only hyperbolic elements. 
\end{lemma}
\begin{proof}
Note that a compact hyperbolic manifold cannot have arbitrarily short non-trivial loops. 
So to prove the lemma it suffices to prove that any parabolic element $\phi\in\mathrm{Aut_{hol}}(\mathbb B^2)$
identifies points with arbitrarily small Kobayashi distance between them. We demonstrate this in the Siegel upper 
half plane $H_2=\{\mathrm{Im}(w)>|z|^2\}$, and up to conjugation there are two cases to consider 
$$
\phi(z,w) = (e^{i\theta}z,w+1) \mbox{ and } \phi(z,w)=(z-i,w-2z + i). 
$$
In the first case we consider points $a_s=(0,i\cdot s)$ and $\phi(a_s)=(0,is+1)\subset \{0\}\times \{\mathrm{Re}(w)>0\}=:H^0_2$. 
In $H_2^0$ the Kobyashi metric is given by $\frac{|dw|^2}{\mathrm{Im}(w)^2}$, and so it is clear that $\mathrm{dist}_K(a,\phi(a_s))\rightarrow 0$
as $s\rightarrow\infty$.  \

In the second case we consider points $a_s=(i,is), s>1$, and we have that $\phi(a_s)=(0,i(s-1))$.  Then 
to estimate the distance between $a_s$ and $\phi(a_s)$ we connect the two points by joining two paths $\gamma^s_1$
and $\gamma^s_2$, the first one being the straight line segment between $i(s-1)$ and $is$ in $H_2^0$, and 
the second being the straight line segment between $(0,is)$ and $(i,is)$ inside the complex disk 
$$
D_s=\{(z,w)\in H_2: w=is\}. 
$$
Then, by the the formula for the Kobayashi metric in $H_2^0$ above, it is clear that the length of $\gamma_1$
goes to zero as $s\rightarrow\infty$. So we consider $\gamma^s_2$. Then 
$$
D_s=\{(z,is):|z|^2<s\},
$$
and so it is clear that the Kobayashi length of $\gamma^s_2$ in $D_s$ goes to zero as $s\rightarrow\infty$. 
\end{proof}

$\hfill\square$

\medskip

\begin{definition}
We now extend our definition of a Fuchsian group to include certain sub-groups of $\mathrm{Aut_{hol}}(\mathbb B^2)$.
For a Fuchsian group $\Gamma\subset\mathrm{Aut_{hol}}(\triangle)$ we may extend each element 
$$
\varphi(z)=e^{i\theta}\frac{z+\alpha}{1-\overline\alpha z}
$$
to an element
$$
\phi_\varphi(z,w)=(e^{i\theta}\frac{z+\alpha}{1-\overline\alpha z},e^{i\psi_\varphi}\frac{\sqrt{1-|\alpha|^2}}{1-\overline\alpha z}w)
$$
of $\mathrm{Aut_{hol}}(\mathbb B^2)$, and it is easy to see that the group $\tilde\Gamma$
generated by the $\phi_\varphi$s is a Kleinian group. We will refer to such special Kleinian groups as 
$2$-dimensional Fuchsian groups.
\end{definition}

\begin{theorem}\label{2fuchsian}
Let $\Gamma$ be a $2$-dimensional Fuchsian group, and set $X=\mathbb B^2/\Gamma$. Then one of the two following
cases can occur. 
\begin{itemize}
\item[(1)] $\Gamma$ is generated by a single parabolic $\phi_\varphi\in\mathrm{Aut_{hol}}(\mathbb B^2)$, $\psi_\varphi=0$, and $\omega^X_\iota=\omega^X_K$.
\item[(2)] $\omega^X_\iota\neq\omega^X_K$.
\end{itemize}
\end{theorem}
\begin{proof}
We consider first the case where $\Gamma$ is not generated by a single parabolic element.  
Then if $\Gamma$ contains a hyperbolic element we have that (2) follows from the previous theorem. 
Suppose then that $\Gamma$ contains extensions of two parabolic elements $\gamma$ and $\tilde\gamma$, 
representing two distinct generators of the fundamental group of the quotient.  
We will show that then $\langle \gamma,\tilde\gamma\rangle$
contains a hyperbolic element. For this it is more convenient to pass to the upper half plane, where first of all,  
up to conjugation, we may assume that $\gamma(z)=z+1$. Furthermore, any automorphism of the upper half plane is on the form 
\begin{equation}\label{normalform}
\tilde\gamma(z)=\frac{az+b}{cz+d} \mbox{ with } a,b,c,d\in\mathbb R, ad-bc=1,
\end{equation}
and moreover, an automorphism of the form \eqref{normalform} is parabolic 
if and only if $a+d=\pm 2$. Now $Y=H/\langle \gamma,\tilde\gamma\rangle$ is an open Riemann surface, and so 
it is known that  $\langle \gamma,\tilde\gamma\rangle$ is a free group, and so the two elements 
do not commute. So we may assume that $c\neq 0$.

Suppose now that $\gamma\circ\tilde\gamma$
is parabolic. We have that 
$$
\gamma(\tilde\gamma(z)) = \frac{az + b}{cz+d} + 1 = \frac{(a+c)z+b+d}{cz+d},
$$
which is parabolic if and only if $a+c+d=\pm 2$. If $a+d=2$ this would imply 
that $c=-4$. But then 
$$
\gamma^{-1}(\gamma(z)) = \frac{(a-c)z+b-d}{cz+d}
$$
is not parabolic, since $a-c+d=6$. The analogous argument applies if $a+d=-2$. \

Suppose next that $\Gamma$ is generated by a single parabolic element $\gamma$. It 
is then more convenient to consider the Siegel upper half-plane model for the unit ball $H_2$
where we up to conjugation have that 
$$
\gamma(z,w)=(e^{i\theta}z,w+1).
$$
Suppose first that $\theta\neq 2k\pi, k\in\mathbb Z$. Setting 
$$
L_{a,b} = \{(z,w): z + aw + b = 0\}
$$
we have that 
$$
G_{a,b} = L_{a,b}\cap H_2
$$
are unique geodesics in $H_2$. We are looking for $G_{a,b}$ that will allow us to apply 
Proposition \ref{prop:genstrat}, i.e., such that $G_{a,b}\cap\gamma(G_{a,b})$
is a single point. So we consider the set of equations 
$$
z + aw + b = 0
$$
and 
$$
e^{i\theta}z + a(w+1) + b = 0.
$$
We set $w=-\frac{b+z}{a}$ and solve
$$
e^{i\theta}z + a  -b - z + b = 0,
$$
and we see that we may set $z_0=\frac{a}{1-e^{i\theta}}$ and then $w_0=-\frac{b}{a} + \frac{1}{1-e^{i\theta}}$
to get that $(z_0,w_0)=G_{a,b}\cap\gamma(G_{a,b})$. Finally, note that for any $a$ one may choose 
$b$ such that $(z_0,w_0)\in H_2$. Now Proposition \ref{prop:genstrat} applies. \

It remains to consider the case that $\theta=0$. In that case, the above calculations show that
the only possibility to achieve that $G_{a,b}\cap\gamma(G_{a,b})\neq\emptyset$ is to 
set $a=0$, i.e., to consider straight vertical lines $G_b=G_{0,b}$. In that case, 
we have that $G_b$ is invariant under $\gamma$, we have that $Z=G_b/\langle\gamma\rangle$
is conformally the punctured disk, and $Z\hookrightarrow X$ is a closed submanifold. Now,
if we consider the covering map $\pi=H_2\rightarrow X$ restricted to $G_b$, 
we have that $\pi:G_b\rightarrow Z$ is a universal covering map, so in this case it 
is not a priori clear if anything prevents $\pi|_{G_b}$ from being a uniform limit (on compacts) of 
injective holomorphic embeddings. In fact, we will show that it is. \

\medskip

Let $H_b=\{\zeta\in\mathbb C:\mathrm{Im}(\zeta)>|b|\}$.
Set $g_b(\zeta)=(-b,\zeta)$ so that $g$ maps $H_b$ onto $G_b$, and 
set $f_b=\pi\circ g_b$. Then $f_b$ is an extremal map for any point in $H_b$, and 
we will pick an arbitrary point $\zeta_0\in H_b$, and show that for any compact 
subset $K\subset H_b$, we have havt that $f_b$ may 
be approximated arbitrarily well on $K$ by injective holomorphic embeddings 
$\tilde f_b:K\rightarrow X$, with the additional property that $d\tilde f_b(\zeta_0)(1)=df_b(\zeta_0)(1)$ are co-linear.  \

By Siu's theorem we have that $f_b(H_b)$ has a Stein neighbourhood $\Omega\subset X$.
Choose local coordinates near $f_b(\zeta_0)$ such that, in the local coordinates in $\mathbb C^2$, we 
have that  $f_b(\zeta_0)=0$ and $df_b(\zeta_0)(1)=(1,0)$.  Since $\Omega$ is Stein 
there are holomorphic vector fields $V_1$ and $V_2$ on $\Omega$ such that in the local 
coordinates just chosen, we have that 
$$
V_1(z)=z_1\frac{\partial}{\partial z_2}+O(\|z\|^2) \mbox{ and } V_2(z) = z_1\frac{\partial}{\partial z_1} + O(\|z\|^2),
$$
and we get the, locally near the origin, the flows are given by $\psi^1_t(z)=(z_1,z_2 + tz_1) + O(\|z\|^2)$
and $\psi^2_t(z)=(e^tz_1,z_2) + O(\|z\|^2)$. \

Now, for a compact set $K\subset H_b$ we let $W\subset\subset X$ be a neighborhood of $f_b(K)$ 
such that the composition of flows $\psi^2_{t_2}\circ\psi^1_{t_1}$ exists on $W$ for 
$|t_1|,|t_2|<\epsilon$ for some $\epsilon>0$. Now for $\delta>0$ we set 
$g_b^\delta(\zeta)=(-b + \delta (\zeta-\zeta_0),\zeta)$, and further $f^\delta_b=\pi\circ g_b^\delta$.
Then $f^\delta_b:K\rightarrow X$ is injective since $g_b^\delta$ maps 
$H_b$ onto a non-vertical line, and $g^\delta_b\rightarrow g_b$ uniformly on $K$ as $\delta\rightarrow 0$.
Now in the local coordinates we have that $df^\delta_b(\zeta)(1)=(1+\eta(\delta),\mu(\delta))$
where $\eta(\delta),\mu(\delta)\rightarrow 0$ as $\delta\rightarrow 0$. Now provided $\delta$
is sufficiently small we may set $s_1(\delta)=-\frac{\mu(\delta)}{1+\eta(\delta)}$ and $s_2(\delta)=\log((1+\eta(\delta))^{-1})$
and set 
$$
\tilde f^\delta_b =    \psi^2_{s_2(\delta)}\circ\psi^1_{s_1(\delta)}\circ f^\delta_b,
$$
and we get that $d\tilde f^\delta_b(\zeta)(1)=(1,0)$ for all $\delta$ (small) and $\tilde f^\delta_b\rightarrow f_b$
uniformly on $K$ as $\delta\rightarrow 0$.

\end{proof}

\section{Quotients of the bi-disk}\label{sec:bidisk}

In this section we will give two results on quotients of the bi-disk. A reason why the case of a bi-disk is more involved than 
the simple case of the unit ball, is that the extremal disks are not unique. Hence, Proposition \ref{prop:genstrat} cannot 
be applied to any extremal holomorphic disk, and we have to work with the "diagonals" $D_\xi=\{(z,z\cdot\xi)\}$ with $|\xi|=1$. 

Theorem \ref{thm:bidisk1} is a consequence 
of the following. 
\begin{theorem}\label{thm:bidisk2} (Jarnicki \cite{Jarnicki1})
Let $\Gamma_j\subset\mathrm{Aut_{hol}}(\triangle)$ be  
Fuchsian groups for $j=1,2$, and set $X:=\triangle^2/\Gamma$, with $\Gamma=\Gamma_1\oplus\Gamma_2$.
Suppose that there exists at least one element $\phi=(\varphi_1,\varphi_2)\in\Gamma$ such that 
$\varphi_j\neq\mathrm{id}$ for $j=1,2$. Then $\omega^X_\iota\neq\omega^X$.
\end{theorem}

The reason why Theorem \ref{thm:bidisk1} follows from this, is that if $\Gamma$ only contains 
elements of the form either $(\varphi,\mathrm{id})$ or $(\mathrm{id},\varphi)$, then $X$ would 
not be compact. \

\medskip

Our next result is positive. 

\begin{theorem}\label{thm:bidisk3}
Let $\Gamma\subset\mathrm{Aut_{hol}}(\triangle^2)$ be a Kleinian group, and set $X:=\triangle^2/\Gamma$.
Suppose all elements $\phi\in\Gamma$ are of the form $(\varphi,\mathrm{id})$. Then if $(\triangle\times\{0\})/\Gamma$
is an open Riemann surface we have that $\omega^K_\iota=\omega^K$.
\end{theorem}

The main step in proving Theorem \ref{thm:bidisk2} is to prove it in the special case that $\Gamma=\Gamma_1\oplus\Gamma_2$ 
where the $\Gamma_j$s are cyclic groups in $\mathrm{Aut_{hol}}(\triangle)$, i.e., when $X=Y_1\times Y_2$ where 
the $Y_j$ are hyperbolic Riemann surfaces with $\pi_1(Y_j)=\mathbb Z$. 
We will consider three cases separately, and then we will explain the general case.

\subsection{The case where the $Y_j$'s are both annuli}

In this section we will prove the following:
\begin{theorem}\label{annuli}
Let $Y_1$ and $Y_2$ be annuli, and set $X:=Y_1\times Y_2$. Then $\omega^X_\iota\neq\omega^X_K$.
\end{theorem}
To prove this theorem we include a subsection where we introduce a way of measuring conformal moduli of annuli, and then 
we give the proof in the subsection following it.

\subsection{Conformal moduli of annuli}

Suppose $X$ and $Y$ are Riemann surfaces with a biholomorphism $\psi:X\rightarrow Y$.
Then a choice of lifting of $\psi(\pi(0))$ induces an automorphism $g\in\mathrm{Aut_{hol}}(\triangle)$
such that the following diagram commutes. 
 
\[\begin{tikzcd}
\triangle \arrow{r}{g} \arrow[swap]{d}{\pi} & \triangle \arrow{l}{}\arrow{d}{\pi'} \\
X \arrow{r}{\psi} & Y\arrow{l}{}
\end{tikzcd}
\]
Then if $\Gamma$ denotes the Deck-group associated to $\pi$ we have that 
$\Gamma'=g\Gamma g^{-1}$ is the Deck group associated to $\pi'$.
On the other hand, if $\Gamma$ is a Fuchsian group, if $\pi:\triangle\rightarrow X=\triangle/\Gamma$ is 
the universal cover, and if $g\in\mathrm{Aut_{hol}}(\triangle)$, then $g$ induces a biholomorphism 
$\psi:X\rightarrow Y$, where $Y=\triangle/\Gamma'$ with $\Gamma'=g\Gamma g^{-1}$. Conjugating by 
such a $g$ corresponds to a change of basepoint and direction for the universal covering map.  \

Now suppose $X$ is an annulus, i.e., that $X=\triangle/\langle\varphi\rangle$ where 
$\varphi$ is hyperbolic. Then $\varphi$ has precisely two fixpoints $p^\alpha$ and $p^\rho$ on $b\triangle$,  
one attracting and one repelling, and we let $\lambda^\alpha$ and $\lambda^\rho$ denote their multipliers. 
Furthermore, $p^\alpha$ and $p^\rho$ are joined by the closure of a unique geodesic $\gamma\subset\mathbb D$.
After conjugation (change of base point) we may assume that $p^\alpha=$ and  $p^\rho =-1$, in which case 
$\gamma=\mathbb R\cap\triangle$, we have that $\gamma$ is $\varphi$-invariant, and  
we 
have that $\varphi$ is on the form 
\begin{equation}\label{rform}
\varphi(z)=\frac{z+r}{1+rz}.
\end{equation}
Then $1$ is an attracting fixed point for $\varphi$, and $\varphi^n\rightarrow 1$ uniformly on compact 
subsets of $\overline\triangle\setminus\{-1\}$ as $n\rightarrow\infty$, so $\gamma$ is the unique
$\varphi$-invariant geodesic in $\triangle$. Since the multipliers are invariant under conjugation we 
may compute them directly from the form \eqref{rform} and we see that 
$\lambda^\alpha=1-r$ and $\lambda^\rho=1+r$.
It follows that $r$ is completely determined by any one of the multipliers, that the multipliers are 
determined completely by $r$, and 
we set 
$$
\mathcal M(X)=\frac{1}{2}\log\frac{1+r}{1-r}=l_K(\gamma).
$$
\begin{prop}\label{modulus}
Let $X$ and $Y$ be two annuli. Then $X$ is biholomorphic to $Y$ if and only if $\mathcal M(X)=\mathcal M(Y)$. 
\end{prop}
\begin{proof}
Suppose $\mathcal M(X)=\mathcal M(Y)$. Then after conjugation we may assume that both $X$ and $Y$
are quotients of the disk by the group generated by the same map \eqref{rform}, and so they are biholomorphic.  \

Next suppose $X$ is biholomorphic to $Y$. After conjugation we may assume that both $X$ and 
$Y$ are quotients of the disk by the groups generated by maps $\varphi_{r_X}$ and $\varphi_{r_Y}$ on the form \eqref{rform}, but this 
time a priori with different dilations $r_X$ and $r_Y$. As noted above, the biholomorphism induces
a conjugation $g$, which necessarily has to fix the points $\pm 1$ individualy, and so $g$ is also on the 
form \eqref{rform}. Consider what happens for the conjugation 
$$
\varphi_{r_Y} = g\circ\varphi_{r_X}\circ g^{-1}
$$
at the fixed point $1$. By the chain rule, the map $g\circ\varphi_{r_X}\circ g^{-1}$ has the same multiplier 
as the map $\varphi_{r_X}$, and so the map $\varphi_{r_X}$ has the same multiplier 
as the map $\varphi_{r_Y}$, from which it follows that $r_X=r_Y$. 
\end{proof}


Next we would like to establish a growth of length description for certain families 
of non-trivial loops in $X$, and that that the loop $\gamma$ in the definition of $\mathcal M(X)$ is 
in fact the shortest non-trivial loop in $X$. 

\begin{prop}\label{increase}
Let $\phi(z)=\frac{z+r}{1+rz}$ and fix $\theta\in (0,\pi)\cup (\pi,2\pi)$. For $s\in [0,1]$ set 
$$
\eta(s):=d_K(se^{i\theta},\phi(se^{i\theta})).
$$
Then $\eta(s)$ is strictly increasing in $s$ and $\lim_{s\rightarrow 1} \eta(s)=\infty$.
\end{prop}
\begin{proof}
Letting $d_M(\cdot,\cdot)$ denote the M\"{o}bius distance, we may prove that 
$$
\lim_{s\rightarrow 1} = \tau(s) := d_M^2(se^{i\theta},\phi(se^{i\theta}))=1,
$$
and that $\tau$ is an increasing function of $s$.
We have that 
$$
\phi(se^{i\theta}) = \frac{se^{i\theta} + r}{1+rse^{i\theta}}, 
$$
and further we get that
\begin{align*}
d_M(se^{i\theta},\phi(se^{i\theta})) & = |\frac{\frac{se^{i\theta} + r}{1+rse^{i\theta}} - se^{i\theta}}{1 - se^{-i\theta} \frac{se^{i\theta} + r}{1+rse^{i\theta}}}|\\
& = |\frac{se^{i\theta} + r - se^{i\theta}(1+rse^{i\theta})}{(1+rse^{i\theta}) - se^{-i\theta}(se^{i\theta} + r)}| \\
& = |\frac{r(1-s^2e^{2i\theta})}{1 - s^2 + rs(e^{i\theta} - e^{-i\theta})}|, 
\end{align*}
so we see that 
$$
\lim_{s\rightarrow 1} d_M(se^{i\theta},\phi(se^{i\theta}))  = |\frac{(1-e^{2i\theta})}{e^{i\theta}-e^{-i\theta}}| = 1. 
$$
Further we have that 
\begin{align*}
\tau(s) & = r^2\frac{(1-s^2\cos(2\theta))^2 + s^4\sin^2(2\theta)}{(1-s^2)^2 + 4r^2s^2\sin^2(\theta) }\\
& = r^2\frac{1 - 2s^2\cos(2\theta) + s^4\cos^2(2\theta) + s^4\sin^2(2\theta)}{(1-s^2)^2 + 4r^2s^2\sin^2(\theta)}\\
& = r^2\frac{1 - 2s^2(\cos^2(\theta)-\sin^2(\theta)) + s^4}{(1-s^2)^2 + 4r^2s^2\sin^2(\theta)}\\
& = r^2\frac{1 - 2s^2(1-2\sin^2(\theta)) + s^4}{(1-s^2)^2 + 4r^2s^2\sin^2(\theta)}\\
& =  r^2\frac{1 - 2s^2 + 4s^2\sin^2(\theta)) + s^4}{(1-s^2)^2 + 4r^2s^2\sin^2(\theta)}\\
& =  r^2\frac{(1 - s^2)^2 + 4s^2\sin^2(\theta))}{(1-s^2)^2 + 4r^2s^2\sin^2(\theta)}\\
& = r^2 [1 + \frac{4s^2\sin^2\theta - 4r^2s^2\sin^2\theta}{(1-s^2)^2 + 4r^2s^2\sin^2(\theta)}]\\
& = r^2 [1 + \frac{s^2\cdot (4(1-r^2)\sin^2\theta)}{(1-s^2)^2 + 4r^2s^2\sin^2(\theta)}]
\end{align*}
So $\tau(s)$ is strictly increasing in $s$ if the function 
$$
f(x)= \frac{x}{(1-x)^2 + x\alpha}
$$
is strictly increasing for $\alpha>0$. Computing the nominator $N(f'(x))$ we see that 
$$
N(f'(x)) = (1-x)^2 + x\alpha - x(2(1-x)(-1) + \alpha) = (1-x)^2 + 2x(1-x)
$$
which is strictly positive for $0\leq x <1$.
\end{proof}

\begin{corollary}
For an annulus $X$ we have that $\mathcal M(X)$ is the Kobayashi length of the shortest non-trivial loop in $X$.
\end{corollary}

\begin{proof}
Choose the universal covering map $\pi:\triangle\rightarrow X$ such that Deck($X$) is generated by 
$$
\varphi(z)=\frac{z+r}{1+rz}.
$$
Let $\gamma:[0,1]\rightarrow X$ be a continuous map with $\gamma(0)=\gamma(1)=p$, 
and let $\tilde\gamma:[0,1]\rightarrow\triangle$ be a lifting of $\gamma$. Assuming that $\gamma$
is a candidate for a shortest non-trivial loop, we may assume that $\tilde\gamma$ is a geodesic
arc in $\triangle$ connecting $\tilde\gamma(0)$ and $\tilde\gamma(1)$. Assume first that $\tilde\gamma(0)\notin\mathbb R$
(in which case $\tilde\gamma(1)\notin\mathbb R$). Write $se^{i\theta}$, and fix $n\in\mathbb Z$ such that $\tilde\gamma(1)=\varphi^{n}(se^{i\theta})$.
Then 
$$
\varphi^{n}(z)=\frac{z+r'}{1+r'z}
$$
for some $r'$ with $|r'|\geq |r|$, and applying the proposition with $\varphi^n$, we see that the Kobayashi length of $\tilde\gamma$
is strictly longer than the segment between $0$ and $\varphi^n(0)$, which in turn has length $\mathcal M(X)$ if and only if $n=\pm 1$. \

In the remaining case, if $\tilde\gamma(0)\in\mathbb R$ then $\tilde\gamma\subset\mathbb R$ and $\tilde\gamma(1)=\varphi^n(\tilde\gamma(0))$
for $n\in\mathbb Z$. By the minimality assumption we have that $n=\pm 1$, in which case $\tilde\gamma$ has the same Kobayashi length 
as the line segment between $0$ and $\varphi(0)$. 

\end{proof}

\subsection{Proof of Theorem \ref{annuli}}
We start by providing a lemma. Recall that $D\subset\triangle^2$ denotes the diagonal $\{(z,z)\}$.
For any point $p\in D$ and $\delta>0$ we let $D_p(\delta)$ denote the set 
$B_\delta(p)\cap D$, where $B_\delta(p)$ denotes the ball of radius $\delta$
centred at $p$ in $\triangle^2$.

\begin{lemma}\label{diag}
Let $\Gamma\subset\mathrm{Aut_{hol}}(\triangle^2)$ be a Fuchsian group, 
assume that $\phi=(\varphi_1,\varphi_2)\in\Gamma$, with $\varphi_1(0)=\varphi_2(0)$, 
and $\varphi_1\neq\varphi_2$, and consider the quotient 
$\pi:\triangle^2\rightarrow X:=\triangle^2/\Gamma$. Then $\pi(D)$ is a singular
curve in $X$, and there exists a $\delta>0$ such that $\pi(D_\delta(0))$ and 
$\pi(D_\delta(\phi(0)))$ are distinct (locally) irreducible components of $\pi(D)$.
\end{lemma}
\begin{proof}
Set $q=\phi(0)$, and 
choose $\delta$ sufficiently small such that $\pi$ is injective on $B_\delta(0)$ and $B_\delta(q)$.
Then $\pi(D_\delta(0))$ and $\pi(D_\delta(q))$ are smooth subsets of $\pi(D)$ and they 
intersect at the point $\pi(0)=\pi(q)$. Consider the set 
$$
Z=\{(z,z)\in D_\delta(0):\pi(z,z)\subset\pi|_{D_\delta(q)}(D_\delta(q))\}=\{(z,z)\in D_\delta(0):\phi(z,z)\subset D_\delta(q)\}.
$$
Then $(z,z)\in Z$ if and only if the equation $\varphi_1(z)=\varphi_2(z)$ is satisfied, but since $\varphi_1\neq\varphi_2$
we have that $0$ is an isolated point satisfying this equation, and the conclusion 
of the lemma follows after possibly having to decrease $\delta$. 
\end{proof}

Suppose first that $Y_1$ is not conformally equivalent to $Y_2$. According to Proposition \ref{modulus} we have that 
$\mathcal M(Y_1)\neq\mathcal M(Y_2)$, and so without loss of generality we assume that $\mathcal M(Y_1)<\mathcal M(Y_2)$.
Let $\pi_j:\triangle\rightarrow Y_j$ be universal covering maps for $j=1,2$, and let $\phi_j$ be generators for the 
corresponding Deck-groups. After conjugating $\phi_j$ for $j=1,2$, and possibly taking inverses, we may assume that 
\begin{equation}\label{form}
\phi_j(z)=\frac{z+r_j}{1+r_jz}.
\end{equation}
(Note that conjugating groups corresponds to changing the base points for $\pi_j$ and a choice of directional 
derivative for the universal covering map.)
We then have that $r_1<r_2$. By Lemma \ref{increase} there exists 
a point $z\in\triangle$ such that $d_K(z,\phi_1(z))=\mathcal M(Y_2)$.  \

Now let $C_1$ denote the straight line segment between -1 and 1, and let $C_2$ be 
the geodesic in $\triangle$ that contains $z$ and $\phi_1(z)$. Then there is a (unique)
M\"{o}bius transformation $\psi$ that maps $C_1$ onto $C_2$, and with $\psi(0)=z, \psi(r_2)=\phi_1(z)$.
Set $\tilde\phi_1=\psi^{-1}\circ\phi_1\circ\psi$. Then $\tilde\phi_1(0)=r_2=\phi_2(0)$. 
However, note that $\tilde\phi_1\neq\phi_2$, since $C_2$ is not an invariant geodesic for $\phi_1$, which implies 
that $C_1$ is not an invariant geodesic for $\tilde\phi_1$.  \

We now consider the universal covering of $X$ given by 
$$
\pi:\triangle^2\rightarrow \triangle^2/\langle \tilde\phi_1(z_1),\phi_2(z_2)\rangle.
$$
Since $\tilde\phi_1(0)=\phi_2(0)=r_2$ it follows from Lemma \ref{diag} that the 
diagonal $D\subset\triangle\times\triangle$ is mapped onto a singular 
locally reducible curve in $X$. Moreover, setting $q=(\tilde\phi_1(0),\phi_2(0))$, near the point $\pi(0)=\pi(q)$
we have that $\pi(D_\delta(0))$ and $\pi(D_\delta(q))$ (see notation from Lemma \ref{diag}) are
(local) irreducible components of $\pi(D)$.  
This concludes the proof in the case that $\mathcal M(Y_1)\neq\mathcal M(Y_2)$ by an application 
of Proposition \ref{prop:genstrat}.

\medskip

In the remaining case, after an initial conjugation, we may assume that $\phi_1=\phi_2$, and that they 
are both on the form \eqref{form}. We may then conjugate $\phi_1$ as before to obtain an element 
$\tilde\phi_1$ such that $\tilde\phi_1(0)=\phi_2^2(0)$, while $\tilde\phi_1\neq\phi_2$. Again, the conclusion is 
that the diagonal $D$ decends to a singular curve in the quotient $X$.

\subsection{The case where $Y_1$ is an annulus and $Y_2$ is the punctured disk}

In this case we have generators $\phi_1$ and $\phi_2$ where 
$\phi_1$ is hyperbolic and $\phi_2$ is parabolic. Then 
for any $0<r<\infty$ we have that there exists a point $z\in\triangle$
such that $d_M(z,\phi_2(z))=r$. Pick a point 
$z_0$ such that $d_M(z_0,\phi_2(z_0))=d_M(0,\phi_1(0))$. Choose 
a map $\gamma$ such that $\gamma(0)=z_0$ and 
set $\tilde\phi_2(z)=\gamma^{-1}(\phi_2(\gamma(z)))$. Then 
$|\phi_1(0)|=|\tilde\phi_2(0)|$, and so after another 
conjugation we may assume that $\phi_1(0)=\tilde\phi_2(0)$. 
Since $\phi_1$ is hyperbolic 
and $\tilde\phi_2$ is parabolic, we have that $\phi_1\neq\tilde\phi_2$, and 
so the proof is concluded as in the previous case.

\subsection{The case where the $Y_j$'s are both the punctured disk}{}\

The punctured disk is the quotient of the upper half plane by a cyclic group 
generated by a parabolic element. Any such element is conjugate 
to an element $z\mapsto z \pm 1$, and so we may initially assume that 
$\phi_1(z)=z+1$ and  $\phi_2(z)=z-1$ (if necessary we may also use inverses). 

\begin{lemma}
Set 
$$
\psi(z)=\frac{(7/5)z-(1/5)}{(4/5)z + (3/5)}.
$$
Then $\psi$ is conjugate to $\phi_2$.
\end{lemma}

\begin{proof}
We have  $\psi(1/2)=1/2$.  Set $\gamma(z)=\frac{z-2}{2z}$ and $\gamma^{-1}(z)=\frac{-1}{z-(1/2)}$.

\begin{align*}
\gamma^{-1}(\psi(\gamma(z))) & = \frac{-1}{ \frac{7(\frac{z-2}{2z}) - 1}{4(\frac{z-2}{2z}) + 3} - (1/2)}\\
& = \frac{-1}{ \frac{(\frac{7z-14}{2z}) - \frac{2z}{2z}}{(\frac{4z-8}{2z}) + \frac{6z}{2z}} - (1/2)}\\
& = \frac{-1}{ \frac{(\frac{5z-14}{2z})}{(\frac{10z-8}{2z})} - (1/2)}\\
& = \frac{-1}{\frac{5z - 14}{10z -8} - (1/2)}\\
& = \frac{-10z + 8}{5z - 14 - 5z + 4}\\
& = z-4/5.
\end{align*}
So $(5/4)\gamma^{-1}(\psi(\gamma((4/5)z)))=z-1$. 
\end{proof}

By the lemma, after conjugation we may assume that the two 
groups are generated by $\phi_1$ and $\psi$. Then $\phi_1(i)=\psi(i)$ while 
$\phi_1\neq\psi$.  So the proof is concluded as in the first case. 

\subsection{Proof of Theorem \ref{thm:bidisk3}}

We will consider how the extremal curves in $\triangle^2$ decend to $X$. 
So let $(\alpha,\beta)\in \triangle^2$ be an arbitrary point, and let 
$v\in\mathbb C^2\setminus\{0\}$. Note first that if $v$ is vertical 
then the vertical line through $(\alpha,\beta)$ is mapped injectively 
into $X$, so we have an injective extremal curve. 
Next we assume that $v$ is not vertical and not horizontal. 
By conjugating 
$\Gamma$ we may assume that $\alpha=0$.  
Now let $\psi$ be an automorphism of $\triangle$ such that $\psi(\beta)=0$, 
set $F(z_1,z_2):=(z_1,\psi(z_2))$, and set $\tilde v:=F_*(0,\beta)(v)$.
Then there is an extremal map $(\xi,\lambda\xi)$ through the origin 
in the direction $\tilde v$, and so the curve 
$(\xi,\psi^{-1}(\lambda\xi))$ is extremal in the direction $v$ at $(0,\beta)$.
Now assume two points on this extremal curve are identified by $\Gamma$, i.e., 
there is a point $(\xi,\psi^{-1}(\lambda\xi))$ and an element $\varphi\in\Gamma$
such that $(\varphi(\xi),\psi^{-1}(\lambda\xi))$ equals
$(\varphi(\xi),\psi^{-1}(\lambda\varphi(\xi)))$. Then $\lambda\xi=\lambda\varphi(\xi)$, 
but then $\varphi=\mathrm{id}$ since $\Gamma$ is fixed point free. \

It remains to consider horizontal directions, and this is done in the same way as 
the last part of the proof of Theorem \ref{2fuchsian}.

\section{Examples}\label{sec:ex}

\subsection{The case of dimension one}
In complex dimension one we have that the 
two metrics are the same on $\triangle, \mathbb C$ and 
$\mathbb P^1$. On $\mathbb C^*$ they are different 
because of the Koebe-$\frac{1}{4}$ theorem, which 
also gives that they are different on any torus $T$. On all other 
Riemann surfaces they are different. 

\subsection{Some easy cases in dimension two}

In complex dimension two we have that the two metrics 
coincide on $\mathbb C^2, \mathbb C^*\times\mathbb C, \mathbb C^*\times\mathbb C^*, \mathbb C^*\times\mathbb P^1, \mathbb C^*\times\mathbb P^1$ and $\mathbb P^1\times\mathbb P^1$. It is unknown if they agree on $T\times\mathbb P^1$ where 
$T$ is a torus. The metrics also always coincide on any manifold $X$ with the Density Property, since
for any point $x\in X$ there is a Fatou-Bieberbach domain $\Omega\subset X$ with $x\in \Omega$.

\subsection{The case of dimension greater than two}

If $X$ is a complex manifold of dimension $\mathrm{dim}(X)\geq 3$ we have that $\omega^X_K=\omega^X_\iota$
due to transversality. 

\subsection{Convex domains}

It is a consequence of Lempert's theory that the two metrics always coincide on a bounded strictly convex domain 
$\Omega\subset\mathbb C^2$ with boundary of class $C^3$ (\cite{Lempert1}, \cite{Lempert2}).

\subsection{The symmetrized bi-disk}

The symmetrized bi-disk $\mathbb{G}$ gives an example of a non-convex domain for which the two metrics coincide. 
Agler and Young showed that every two points in $\mathbb{G}$ can be joint by a unique complex geodesic for $\omega^{\mathbb{G}}_K$ that has a left inverse (\cite{AglerYoung} ). This example we might generalise as follows:

\begin{prop}
Let $D\subset\mathbb C^n$ be a bounded taut domain such that for any two points $z_1,z_2\in\Omega$
there exist a holomorphic map $\varphi:\triangle\rightarrow\Omega$ with $z_1,z_2\in\varphi(\triangle)$, 
and a holomorphic map $\psi:D\rightarrow\triangle$ such that $\psi(\varphi(\zeta))=\zeta$ for all $\zeta\in\triangle$. 
Then $\omega^D_\iota=\omega^D_K$.
\end{prop}
\begin{proof}
Let $z\in D$ and let $v\in\mathbb C^n$.  Let $z_j=z+(1/j)(v)$, $\varphi_j(0)=z, \varphi_j(\zeta_j)=z_j$, 
and $\psi_j(\varphi_j(\zeta))=\zeta$.  Without loss of generality we may assume that $\varphi_j\rightarrow\varphi:\triangle\rightarrow D, \psi_j\rightarrow\psi: D\rightarrow \triangle$
uniformly on compacta.  Then $\psi(\varphi(\zeta))=\zeta$ for all $\zeta\in\triangle$.   So $\varphi$ is a holomorphic embedding, $\varphi'(0)=\lambda v$ for 
some $\lambda\neq 0$, and $\Omega^G_K(z,v)=1/|\lambda|$.
\end{proof}

%
%

\section{Open problems}

\begin{problem}
Determine if the injective Kobayashi metric vanishes identically 
on $\mathbb P^1\times T_1$ and $T_1\times T_2$ (here $T_j$ are tori). 
\end{problem}

\begin{problem}
Let $R$ be a compact hyperbolic Riemann surface, and let $S$ denote either 
the unit disk $\triangle$, the complex plane $\mathbb C$, the Riemann sphere $\widehat{\mathbb C}$, 
or a torus $T$. Set $X=R\times S$. Do we have $\omega^X_\iota=\omega^X_K$?
\end{problem}

\begin{problem}
Let $X$ be an Oka manifold. Do we have $\omega^X_\iota=\omega^X_K$?
\end{problem}

\bibliographystyle{amsplain}

\end{document}